\documentclass[reqno]{amsart}

\allowdisplaybreaks[1]

\usepackage{amssymb}
\usepackage{amsmath}
\usepackage{amsthm}
\usepackage[all]{xy}
\usepackage{tikz}
\usepackage{latexsym}
\usepackage[justification=centering]{caption}
\usepackage{color}
\usepackage{hyperref}
\hypersetup{
 colorlinks=true,
 linkcolor=blue,
 citecolor=blue}

 \tikzset{vertex/.style={fill,circle,inner sep=1.0pt}}

\newtheorem{theorem}{Theorem}[section]
\newtheorem{proposition}[theorem]{Proposition}
\newtheorem{corollary}[theorem]{Corollary}
\newtheorem{lemma}[theorem]{Lemma}
\newtheorem{claim}[theorem]{Claim}
\theoremstyle{definition}
\newtheorem{definition}[theorem]{Definition}
\newtheorem{remark}[theorem]{Remark}
\newtheorem{example}[theorem]{Example}

\newcommand{\poly}{\mathcal{Z}^*_M (\underline{K}, \underline{L})}
\newcommand{\polyp}{\mathcal{Z}^*_M (K, L)}
\newcommand{\pair}{(\underline{K}, \underline{L})}
\newcommand{\facet}[1]{\mathbf{F}_{#1}}
\newcommand{\dl}[2]{\mathrm{dl}_{#1}(#2)}
\newcommand{\polyind}{\mathcal{Z}^*_{I(G)} (I(H), I(H \setminus U))}
\newcommand{\polyindv}{\mathcal{Z}^*_{I(G)} (I(H), I(H \setminus \{v_0\}))}

\title[]{Shellability of Polyhedral Joins of Simplicial Complexes and Its Application to Graph Theory}
\author[]{Kengo Okura}
\address{Osaka Metropolitan University, 1-1 Gakuen-cho, Naka-ku, Sakai, Osaka 599-8531, Japan}
\email{okura.kengo.k35@kyoto-u.jp}

\begin{document}

\begin{abstract}
We investigate the shellability of the polyhedral join $\mathcal{Z}^*_M (K, L)$ of simplicial complexes $K, M$ and a subcomplex $L \subset K$.
We give sufficient conditions and necessary conditions on $(K, L)$ for $\mathcal{Z}^*_M (K, L)$ being shellable. In particular, we show that for some pairs $(K, L)$, $\mathcal{Z}^*_M (K, L)$ becomes shellable regardless of whether $M$ is shellable or not. 
Polyhedral joins can be applied to graph theory as the independence complex of a certain generalized version of lexicographic products of graphs which we define in this paper. 
The graph obtained from two graphs $G, H$ by attaching one copy of $H$ to each vertex of $G$ is a special case of this generalized lexicographic product and we give a result on the shellability of the independence complex of this graph by applying the above results.
\end{abstract}

\maketitle

\section{Introduction}
\label{introduction}
A {\it finite simple graph} $G$ is a pair $G=(V(G), E(G))$ of a finite set $V(G)$ and a set $E(G) \subset \{ e \subset V(G) \ |\ |e| = 2 \}$. We drop adjectives ``finite simple'' and call $G$ a {\it graph}. $v \in V(G)$ is called a {\it vertex} of $G$ and $e \in E(G)$ is called an {\it edge} of $G$. The {\it independence complex} of $G$ is an abstract simplicial complex $I(G)$ defined by
\begin{align*}
I(G) = \{ \sigma \subset V(G) \ |\ \{u,v\} \notin E(G) \text{ for any $u, v \in \sigma$ } \}.
\end{align*}
A simplex of $I(G)$ is called an {\it independent set} of $G$. 

A simplicial complex $K$ is {\it shellable} if its facets can be arranged in a linear order $F_1, F_2, \ldots , F_t$ (which we call a {\it shelling}) in such a way that the subcomplex $\left(\bigcup_{i=1}^{k-1} \langle F_i \rangle \right) \cap \langle F_k \rangle$ is pure and $(\dim F_k -1)$-dimensional for all $k=2, \ldots , t$. 
(Note that this is ``non-pure'' shellability defined by Bj{\"{o}}rner and Wachs \cite{BjornerWachs96, BjornerWachs97}.)
A graph is called {\it shellable} if its independence complex is shellable. 
The shellability (including {\it vertex decomposability}, which is one of the sufficient conditions for a graph being shellable) of graphs has been studied by many researchers, such as 
\cite{HibiHigashitaniKimuraOKeefe15, VantuylVillarreal08, VandermeulenVantuyl17, Woodroofe09}.
In this paper, we focus on the following result by Hibi, Higashitani, Kimura, and O'Keefe \cite{HibiHigashitaniKimuraOKeefe15}.
Here, a graph $G$ is called {\it well-covered} if every maximal independent set of $G$ has the same cardinality.

\begin{theorem}[Based on Hibi, Higashitani, Kimura, and O'Keefe {\cite[Theorem 1.1]{HibiHigashitaniKimuraOKeefe15}}]
\label{motivation}
Let $G$ be a graph on a vertex set $V(G)=\{u_1, \ldots, u_n\}$. Let $k_1, \ldots, k_n \geq 2$ be integers. Then the graph $G'$ obtained from $G$ by attaching the complete graph $K_{k_i}$ to $u_i$ for $i=1, \ldots, n$ is well-covered and shellable.
\end{theorem}

\noindent
Motivated by Theorem \ref{motivation}, we consider the graph $G[H;\{v_0\}]$ defined as follows.
Let $G$, $H$ be graphs and $v_0$ be a vertex of $H$. Define $G[H;\{v_0\}]$ as the graph obtained from $G \sqcup \left( \bigsqcup_{u \in V(G)} H_u \right)$, where $H_u$ is a copy of $H$, by identifying $u \in G$ with $v_0 \in H_u$.
Theorem \ref{motivation} implies that if $H$ is a complete graph, then $G[H;\{v_0\}]$ is well-covered and shellable for any graph $G$.

In this paper, we obtain a necessary and sufficient condition on $H$ and $v_0$ for $G[H;\{v_0\}]$ being well-covered and shellable for any graph $G$.
\begin{theorem}
\label{main theorem}
Let $H$ be a graph and $v_0$ be a vertex of $H$. Then the following two conditions are equivalent.
\begin{enumerate}
\item For any graph $G$, $G[H;\{v_0\}]$ is well-covered and shellable.
\item $H$ is well-covered, both $H$ and $H \setminus \{v_0\}$ are shellable, and for any maximal independent set $\tau$ of $H \setminus \{v_0\}$, there exists $v \in \tau$ such that $\{v_0, v\} \in E(H)$.
\end{enumerate}
\end{theorem}

In order to prove Theorem \ref{main theorem}, we need to investigate the independence complex of $G[H;\{v_0\}]$. With $I(G)$ and $I(H)$, the independence complex $I(G[H;\{v_0\}])$ is described as a {\it polyhedral join}. Polyhedral join is a construction of simplicial complexes introduced by Ayzenberg \cite[Definition 4.2, Observation 4.3]{Ayzenberg13}. It is similar to {\it polyhedral product} $\mathcal{Z}_K (\underline{X}, \underline{A})$, a well-known construction of spaces, where $K$ is a simplicial complex and $(\underline{X}, \underline{A}) = \{(X_u, A_u)\}_{u \in V(K)}$ is a family of pairs of spaces. The definition of polyhedral joins is obtained from the definition of polyhedral products by replacing ``pairs of spaces'' and ``product of spaces'' with ``pairs of simplicial complex and its subcomplex'' and ``join of simplicial complexes'', respectively.
Polyhedral joins appear in previous studies, including when they are called by other names.
Bahri, Bendersky, Cohen, and Gitler \cite[Definition 2.1]{BahriBenderskyCohenGitler15} defined a simplicial complex $K(J)$ for a simplicial complex $K$ on $V(K)=\{1, \ldots, n\}$ and a tuple $J=(j_1, \ldots, j_n)$ of positive integers. 
Using our notation of polyhedral joins, $K(J)$ is denoted by
\begin{align*}
&\mathcal{Z}^*_K (\underline{\Delta^{J-1}}, \underline{\partial \Delta^{J-1}})  \text{, where }
(\underline{\Delta^{J-1}}, \underline{\partial \Delta^{J-1}} ) =\{( \Delta^{j_i-1}, \partial \Delta^{j_i-1} )\}_{i \in \{1, \ldots, m\}}.
\end{align*}
Here, $\Delta^d$ is the $d$-simplex and $\partial \Delta^d$ is its boundary. They obtained the decomposition of polyhedral products, more precisely, {\it moment-angle complexes}, denoted by 
\begin{align*}
\mathcal{Z}_K (D^2,S^1) = \mathcal{Z}_{\mathcal{Z}^*_K (\Delta^1, \partial \Delta^1)} (D^1, S^0) .
\end{align*}
We note that the above observation is mentioned by Vidaurre \cite{Vidaurre18}, who investigated the polyhedral products over polyhedral joins.
Another example is $(j_1, \ldots, j_n)$-expansion of $K$, which is introduced by Moradi and Khosh-Ahang \cite[Definition 2.1]{MoradiKhoshahang16}. It is denoted by
\begin{align*}
\mathcal{Z}^*_K (\underline{\mathrm{pt}^{J}}, \{\emptyset\})  \text{, where } (\underline{\mathrm{pt}^J}, \{\emptyset\}) &=\left\{ \left( {\bigsqcup}_{j_i} \mathrm{pt}, \{\emptyset\} \right) \right\}_{i \in \{1, \ldots, m\}}.
\end{align*}
They studied the shellability and vertex decomposability of expansions. We generalize one of their results \cite[Theorem 2.12]{MoradiKhoshahang16}.

This paper is organized as follows. In Section \ref{preliminaries}, we define terminologies and notations on simplicial complexes and state some basic properties of shellable simplicial complexes. Section \ref{facets of polyhedral joins} provides the definition of polyhedral joins and the explicit description of simplices and facets of polyhedral joins. Section \ref{shellability of polyhedral joins} is the main part of this paper. Here we obtain two necessary conditions and two sufficient conditions for polyhedral joins being shellable, giving counterexamples of converse propositions. Note that some of the results are not relevant to Theorem \ref{main theorem}. Finally, in Section \ref{applications}, we apply the results obtained in Section \ref{shellability of polyhedral joins} to the independence complexes of graphs and prove Theorem \ref{main theorem}.

\section{Preliminaries}
\label{preliminaries}
In the following, for a positive integer $m$, we set $[m]=\{1,2,\ldots, m\}$.

An {\it abstract simplicial complex} $K$ is a collection of finite subsets of a given set $V(K)$ such that
if $\sigma \in K$ and $\tau \subset \sigma$, then $\tau \in K$. In this paper, we drop the adjective ``abstract''.
An element of $K$ is called a {\it simplex} of $K$. An element of $V(K)$ is called a {\it vertex} of $K$.
We suppose that $\{ v \} \in K$ for any $v \in V(K)$. 
We set $\dim K = \max_{\sigma \in K} |\sigma| -1 $, where $|\sigma|$ is the cardinality of $\sigma \subset V(K)$. If $\dim K =d$, then $K$ is called {\it $d$-dimensional}.
A maximal simplex with respect to the inclusion is called a {\it facet} of $K$. $K$ is called {\it pure} if every facet of $K$ has the same cardinality.

$L \subset K$ is called a {\it subcomplex} of $K$ if $L$ is a simplicial complex. In this paper, we call $(K, L)$ a {\it pair of simplicial complexes} if $K$ is a simplicial complex and $L \subset K$ is a subcomplex of $K$. 
For a vertex $v \in V(K)$ of $K$, we define a subcomplex $\dl{K}{v}$ of $K$ by
\begin{align*}
\dl{K}{v} = \{ \sigma \in K \ |\ v \notin \sigma \}.
\end{align*}

Let $\{K_i\}_{i \in [m]}$ be a family of simplicial complexes. We define a simplicial complex $K_1 * \cdots * K_m$, which we call the {\it join} of $K_1, K_2, \ldots$, and $K_m$, by
\begin{align*}
K_1 * \cdots * K_m = \left\{ \sigma \subset \bigsqcup_{i \in [m]} V(K_i) \ \middle|\ 
\sigma \cap V(K_i) \in K_i \text{ for any $i \in [m]$}  \right\}.
\end{align*} 

Let $V$ be a finite set and $F_1, \ldots, F_t \subset V$ be a collection of subsets of $V$ such that $F_i \nsubseteq F_j$ for any $i \neq j$. Define a simplicial complex $\langle F_1, \ldots, F_t \rangle$ on $V$ by
\begin{align*}
\langle F_1, \ldots, F_t \rangle= \{ \sigma \subset V \ |\ \sigma \subset F_i \text{ for some } i \in [t] \}.
\end{align*}

As defined in Section \ref{introduction}, a simplicial complex $K$ is {\it shellable} if its facets can be arranged in a linear order $F_1, F_2, \ldots , F_t$ (which we call a {\it shelling}) in such a way that the subcomplex $\left(\bigcup_{i=1}^{k-1} \langle F_i \rangle \right) \cap \langle F_k \rangle$ is pure and $(\dim F_k -1)$-dimensional for all $k=2, \ldots , t$. Here we state some of the properties of shellable simplicial complexes without proofs.

\begin{lemma}[Based on {\cite[Lemma 2.3]{BjornerWachs96}}]
\label{shellable definition}
Let $K$ be a simplicial complex. An order $F_1, F_2, \ldots, F_t$ of the facets of $K$ is a shelling if and only if for every $i, k$ with $1 \leq i < k \leq t$, there exists an index $j$ with $1 \leq j <k$ and a vertex $x \in F_k \setminus F_i$ such that $F_j \cap F_k = F_k \setminus \{x\}$.
\end{lemma}
\begin{lemma}[Based on {\cite[Lemma 2.6]{BjornerWachs96}}]
\label{rearrangement lemma}
Let $K$ be a shellable simplicial complex. Then there exists a shelling $F_1, F_2, \ldots, F_t$ such that $|F_i| \geq |F_j|$ for any $1 \leq i < j \leq t$.
\end{lemma}
\begin{lemma}[{\cite[Remark 10.22]{BjornerWachs97}}]
\label{join shellable}
The join of two simplicial complexes is shellable if and only if each of the simplicial complex is shellable.
\end{lemma}

For a simplicial complex $X$, we denote the set of all facets of $X$ by $\facet{X}$. 
Note that for a pair $(K, L)$ of simplicial complexes, we have
\begin{align*}
\facet{K} \setminus \facet{L} &= \left\{ \sigma \in \facet{K} \ |\ \sigma \notin L \right\}, \\
\facet{L} \setminus \facet{K} &= \left\{ \tau \in \facet{L} \ |\ \text{there exists $\sigma \in \facet{K} \setminus \facet{L}$ such that $\tau \subsetneq \sigma$} \right\} ,\\
\facet{K} \cap \facet{L} &=\left\{ \sigma \in \facet{K} \ |\ \sigma \in L \right\}.
\end{align*}

\section{Facets of Polyhedral Joins}
\label{facets of polyhedral joins}
The main subject of this paper is {\it polyhedral join}. It is a construction of simplicial complexes, which is based on the definition by Ayzenberg \cite[Definition 4.2, Observation 4.3]{Ayzenberg13}. 
\begin{definition}
Let $M$ be a simplicial complex on $[m]$ and $\pair = \{(K_i, L_i)\}_{i \in [m]}$ be a family of pairs of simplicial complexes. For a simplex $S \in M$, we define a simplicial complex $\pair^{*S}$ by
\begin{align*}
\pair^{*S} = X_1 * X_2 * \cdots * X_m, \ X_i= \left\{
\begin{aligned}
&K_i & &(i \in S) , \\
&L_i & &(i \notin S) .
\end{aligned} \right.
\end{align*}
$\pair^{*S}$ is a subcomplex of $K_1 * \cdots * K_m$. 

$\poly$ is a subcomplex of $K_1 * \cdots *K_m$ defined by
\begin{align*}
\poly = \bigcup_{S \in M} \pair^{*S} 
\end{align*}
(union is taken in $K_1 * \cdots *K_m$).

Let $(K, L)$ be a pair of simplicial complexes. If $K_i = K, L_i = L$ for any $i \in [m]$, we write $\polyp$ instead of $\poly$. 
\end{definition}

We give an explicit description of simplices and facets of $\poly$.
\begin{proposition}
\label{alternative definition}
Let $M$ be a simplicial complex on $[m]$ and $\pair = \{(K_i, L_i)\}_{i \in [m]}$ be a family of pairs of simplicial complexes. 
For $\phi \subset \bigsqcup_{i \in [m]} V(K_i)$, we set
\begin{align*}
\phi_i &= \phi \cap V(K_i) \ (i \in [m]), \\
\overline{\phi} &=\{i \in [m] \ |\ \phi_i \notin L_i \}.
\end{align*}
Then $\phi \in \poly$ if and only if $\phi_i \in K_i$ for any $i \in [m]$ and $\overline{\phi} \in M$.
\end{proposition}
\begin{proof}
Let $\phi \in \poly$. Then $\phi$ is a simplex of $\phi \in \pair^{*S}$ for some $S \in M$. 
So, by the definition of the join, we have
\begin{itemize}
\item $\phi \cap V(K_i) = \phi_i \in K_i$ for any $i \in S$,
\item $\phi \cap V(K_i) = \phi_i \in L_i \subset K_i$ for any $i \notin S$.
\end{itemize}
It follows that $\phi_i \in K_i$ for any $i \in [m]$ and that $\overline{\phi} \subset S$. Since $M$ is a simplicial complex and $S \in M$, we have $\overline{\phi} \in M$.

Conversely, let $\phi \subset \bigsqcup_{i \in [m]} V(K_i)$ be a set such that $\phi_i \in K_i$ for any $i \in [m]$ and $\overline{\phi} \in M$. By the definition of $\overline{\phi}$, $i \notin \overline{\phi}$ implies $\phi_i \in L_i$. So, we have $\phi \in \pair^{* \overline{\phi}} \subset \poly$.
\end{proof}

\begin{proposition}
\label{facet description}
Let $M$ be a simplicial complex on $[m]$ and $\pair = \{(K_i, L_i)\}_{i \in [m]}$ be a family of pairs of simplicial complexes. 
Then $\phi \subset \bigsqcup_{i \in [m]} V(K_i)$ is a facet of $\poly$ if and only if
\begin{itemize}
\item $\phi_i \in \facet{K_i} \cup \facet{L_i}$ for any $i \in [m]$,
\item $\overline{\phi} \in M$, and
\item $\overline{\phi} \cup \{i\} \notin M$ for any $i \in [m]$ such that $\phi_i \in \facet{L_i} \setminus \facet{K_i}$.
\end{itemize}
\end{proposition}
\begin{proof}
Let $\phi$ be a facet of $\poly$. For $k \in [m]$ and $\sigma \in K_k$, define $\phi^{(k, \sigma)} \subset \bigsqcup_{i \in [m]} V(K_i)$ by 
\begin{align*}
\phi^{(k,\sigma)}_i = \left\{
\begin{aligned}
& \phi_i & &(i \neq k) \\
& \sigma & &(i=k).
\end{aligned} \right.
\end{align*} 

First, assume that there exists $\sigma \in K_k$ such that $\phi_k \subsetneq \sigma$ for some $k \in \overline{\phi}$. 
We have $\sigma \notin L_k$ since $\sigma \supset \phi_k \notin L_k$. Therefore, we get $\overline{\phi^{(k,\sigma)}} = \overline{\phi} \in M$. So, $\phi^{(k,\sigma)}$ is a simplex of $\poly$ which satisfies $\phi^{(k,\sigma)} \supsetneq \phi$, a contradiction. Thus, $\phi_k$ is a facet of $K_k$. 

Second, assume that there exists $\tau \in L_k$ such that $\phi_k \subsetneq \tau$ for some $k \notin \overline{\phi}$. 
Since $\phi_k ,\tau \in L_k$, we get $\overline{\phi^{(k,\tau)}}= \overline{\phi} \in M$. So, $\phi^{(k,\tau)}$ is a simplex of $\poly$ which satisfies $\phi^{(k,\tau)} \supsetneq \phi$, a contradiction. Thus, $\phi_k$ is a facet of $L_k$. 

Finally, assume that there exists $j \in [m]$ such that $\phi_j \in \facet{L_j} \setminus \facet{K_j}$ and $\overline{\phi} \cup \{j\} \in M$. Then there must be $\rho \in K_j \setminus L_j$ such that $\phi_j \subsetneq \rho$ since $\phi_j$ is a facet of $L_j$ and is not a facet of $K_j$. We get that $\phi^{(j,\rho)}$ is a simplex of $\poly$ since $\overline{\phi^{(j,\rho)}}=\overline{\phi} \cup \{j\} \in M$. This is a contradiction to the maximality of $\phi$. By the above three arguments, we conclude that a facet $\phi$ of $\poly$ satisfies three conditions in the proposition.

Conversely, suppose that $\phi \subset \bigsqcup_{i \in [m]} V(K_i)$ satisfies three conditions in the proposition.
Since $\phi_i \in \facet{K_i} \cup \facet{L_i} \subset K_i$ for any $i \in [m]$ and $\overline{\phi} \in M$, $\phi$ is a simplex of $\poly$. 
We assume that there exists $\psi \in \poly$ such that $\phi \subsetneq \psi$ and deduce a contradiction. There must exists $k \in [m]$ such that $\phi_k \subsetneq \psi_k$. This means that $\phi_k$ is not a facet of $K_k$, which implies $\phi_k \in \facet{L_k} \setminus \facet{K_k}$. 
It follows that $k \in \overline{\psi}$ since $\phi_k \subsetneq \psi_k$ and $\phi_k$ is a facet of $L_k$. Thus, we get $\overline{\phi} \cup \{k\} \subset \overline{\psi} \in M$, which contradicts to the third condition in the proposition. Therefore, $\phi$ is a facet of $\poly$.
\end{proof}

\begin{remark}
\label{facet description remark}
If $\facet{L_i} \subset \facet{K_i}$ for any $i \in [m]$, then by Proposition \ref{facet description}, $\phi \subset \bigsqcup_{i \in [m]} V(K_i)$ is a facet of $\poly$ if and only if $\phi_i \in \facet{K_i}$ for any $i \in [m]$ and $\overline{\phi} \in M$.
\end{remark}

Before investigating the shellability of $\polyp$, we show a necessary and sufficient condition for $\polyp$ being pure.
\begin{theorem}
\label{pure}
Let $M \neq \{\emptyset\}$ be a simplicial complex which is not a simplex and $K \supsetneq L$ be a pair of simplicial complexes. Then $\polyp$ is pure if and only if
\begin{itemize}
\item $K$ is pure and $\facet{L} \subset \facet{K}$, or
\item $K, L, M$ are pure.
\end{itemize}
\end{theorem}
\begin{proof}
We set $V(M)=[m]$.

Assume that $\polyp$ is pure.
For $S \in \facet{M}$, $\sigma \in \facet{K}$ and $\tau \in \facet{L}$, define $\Phi_S(\sigma, \tau) \subset \bigsqcup_{i \in [m]} V(K)$ by
\begin{align*}
&\Phi_S(\sigma, \tau)_i = \left\{
\begin{aligned}
& \sigma & &(i \in S) \\
& \tau & &(i \notin S).
\end{aligned} \right.
\end{align*}
Here, remark that $\Phi_S(\sigma, \tau)$ is a facet of $\polyp$ if $\sigma \in \facet{K} \setminus \facet{L}$ or $\tau \in \facet{K} \cap \facet{L}$ by Proposition \ref{facet description}. If $\Phi_S ( \sigma, \tau)$ is a facet of $\polyp$, then we have
\begin{align*}
|\Phi_S (\sigma, \tau)| &= \sum_{i \in [m]} |\Phi_S(\sigma, \tau)_i| = |S| |\sigma| + (m-|S|)|\tau|.
\end{align*}
Note that we have $0<|S|<m$ since $M$ is neither $\{\emptyset\}$ nor a simplex.

Suppose that $\facet{K} \cap \facet{L} \neq \emptyset$. Take $S \in \facet{M}$, $\sigma^0 \in \facet{K} \setminus \facet{L}$ and $\tau^0 \in \facet{K} \cap \facet{L}$. For any facets $\sigma, \sigma' \in \facet{K}$ of $K$, $\Phi_S(\sigma, \tau^0)$ and $\Phi_S(\sigma', \tau^0)$ are facets of $\polyp$. Since $\polyp$ is pure, we have $|\Phi_S(\sigma, \tau^0)| = |\Phi_S(\sigma' ,\tau^0)|$, namely
\begin{align*}
|S| |\sigma| + (m-|S|)|\tau^0| =|S| |\sigma'| + (m-|S|)|\tau^0| .
\end{align*}
It follows from $|S|>0$ that
\begin{align*}
|\sigma| = |\sigma'|.
\end{align*}
Therefore, $K$ is pure. For any facets $\tau, \tau' \in \facet{L}$ of $L$, 
$\Phi_S(\sigma^0, \tau)$ and $\Phi_S(\sigma^0, \tau')$ are facets of $\polyp$. Since $\polyp$ is pure, we have $|\Phi_S(\sigma^0, \tau)| = |\Phi_S(\sigma^0 ,\tau')|$, namely
\begin{align*}
|S| |\sigma^0| + (m-|S|)|\tau| =|S| |\sigma^0| + (m-|S|)|\tau'| .
\end{align*}
It follows from $|S|<m$ that
\begin{align*}
|\tau| = |\tau'|.
\end{align*}
Therefore, $L$ is pure. Hence, we obtain
\begin{align*}
\dim K = |\tau^0| -1 = \dim L .
\end{align*}
If there exists $\tau^1 \in \facet{L} \setminus \facet{K}$, then there must exist $\sigma^1 \in \facet{K} \setminus \facet{L}$ such that $\tau^1 \subsetneq \sigma^1$. So, we get
\begin{align*}
\dim K = |\sigma^1| -1 > |\tau^1|-1 = \dim L,
\end{align*}
which is a contradiction. Thus, we conclude that $\facet{L} \setminus \facet{K} = \emptyset$, namely $\facet{L} \subset \facet{K}$.

Next, suppose that $\facet{K} \cap \facet{L} = \emptyset$. Take $S \in \facet{M}$ and $\tau^0 \in \facet{L}$. Since $\tau^0$ is not a facet of $K$, there exist $\sigma^0 \in \facet{K}$ such that $\tau^0 \subsetneq \sigma^0$. For any facet $\sigma$ of $K$, $\Phi_S (\sigma, \tau^0)$ is a facet of $\polyp$ since $\facet{K} = \facet{K} \setminus \facet{L}$. Therefore, by the same argument as above, we obtain that both $K$ and $L$ are pure. For any facet $T, T' \in \facet{M}$ of $M$, $\Phi_T(\sigma^0, \tau^0)$ and $\Phi_{T'} (\sigma^0, \tau^0)$ are facets of $\polyp$. Since $\polyp$ is pure, we have $|\Phi_T (\sigma^0, \tau^0)| = |\Phi_{T'} (\sigma^0 ,\tau^0)|$, namely
\begin{align*}
|T| |\sigma^0| + (m-|T|)|\tau^0| =|T'| |\sigma^0| + (m-|T'|)|\tau^0| .
\end{align*}
It follows that
\begin{align*}
|T| = |T'|
\end{align*}
since $\sigma^0 \supsetneq \tau^0$. Therefore, we conclude that $M$ is pure. 

Conversely, assume that $K$ is pure and $\facet{L} \subset \facet{K}$. Then for any facet $\phi \in \facet{\polyp}$, we get
\begin{align*}
|\phi| = \sum_{i \in [m]} |\phi_i| = m(\dim K +1).
\end{align*}
Therefore, $\polyp$ is pure.

Finally, assume that $K, L, M$ are pure and that there exists $\tau \in \facet{L} \setminus \facet{K}$. Since there exists $\sigma \in \facet{K} \setminus \facet{L}$ such that $\tau \subsetneq \sigma$, we get
\begin{align*}
\dim K = |\sigma| -1 > |\tau|-1 = \dim L.
\end{align*}
So, we obtain $\facet{K} \cap \facet{L} = \emptyset$ since both $K$ and $L$ are pure. Thus, for any $\phi \in \facet{\polyp}$, $\overline{\phi}$ is a facet of $M$. This is because we have $\overline{\phi} \cup \{i\} \notin M$ for any $i \notin \overline{\phi}$ since $\facet{L} = \facet{L} \setminus \facet{K}$. Hence, we get
\begin{align*}
|\phi| & = \sum_{i \in \overline{\phi}} |\phi_i| + \sum_{i \notin \overline{\phi}} |\phi_i| \\
& = |\overline{\phi}|(\dim K +1) + (m- |\overline{\phi}|)(\dim L +1) \\
&=(\dim M +1)(\dim K +1) + (m- \dim M -1) (\dim L +1)
\end{align*}
for any facet $\phi \in \facet{\polyp}$ since $M$ is pure.
Therefore, we conclude that $\polyp$ is pure.
\end{proof}

\section{Shellability of Polyhedral Joins}
\label{shellability of polyhedral joins}
We first present two sufficient conditions for $\polyp$ being shellable.
\begin{theorem}
\label{sufficient}
Let $M$ be an arbitrary simplicial complex and $K \supset L$ be a pair of simplicial complexes. Suppose that $K, L$ satisfy the following three conditions.
\begin{enumerate}
\item $K$ is shellable,
\item $\facet{L} \subset \facet{K}$, and
\item there exists a shelling $<_K$ on $K$ such that $\tau <_K \sigma$ for any $\tau \in \facet{L}$ and any $\sigma \in \facet{K} \setminus \facet{L}$.
\end{enumerate}
Then $\polyp$ is shellable.
\end{theorem}
\begin{proof}
We set $V(M) = [m]$.
Consider a linear order $<$ on $\facet{\polyp}$ such that $\phi < \psi$ if and only if 
\begin{align*}
\phi_1 = \psi_1, \ldots, \phi_{j-1} = \psi_{j-1}, \phi_j <_K \psi_j
\end{align*}
for some $j \in [m]$. We show that this order is a shelling on $\polyp$.

Let $\phi, \psi$ be facets of $\polyp$ such that $\phi_1 = \psi_1, \ldots, \phi_{j-1} = \psi_{j-1}, \phi_j <_K \psi_j$. Since $<_K$ is a shelling on $K$, there exists $\sigma \in \facet{K}$ and $x \in \psi_j \setminus \phi_j$ such that $\sigma <_K \psi_j$ and $\sigma \cap \psi_j = \psi_j \setminus \{x\}$ by Lemma \ref{shellable definition}. Define $\chi \subset \bigsqcup_{i \in [m]} V(K)$ by
\begin{align*}
\chi_i = \left\{
\begin{aligned}
& \psi_i & &(i \neq j) \\
& \sigma & &(i =j).
\end{aligned} \right.
\end{align*}
If $\psi_j \in \facet{K} \setminus \facet{L}$, then we have $\overline{\chi} \subset \overline{\psi} \in M$. 
If $\psi_j \in \facet{L}$, then by condition (3), $\sigma$ must be a facet of $L$ since $\sigma <_K \psi_j$. So, we have $\overline{\chi} = \overline{\psi} \in M$. 
In both cases, we get that $\overline{\chi} \in M$, which implies that $\chi$ is a simplex of $\polyp$. Furthermore, $\chi$ is a facet of $\polyp$ by Remark \ref{facet description remark}.

We have $\chi < \psi$ because
\begin{align*}
\chi_1 =\psi_1, \ldots, \chi_{j-1} = \psi_{j-1}, \chi_j = \sigma <_K \psi_j.
\end{align*}
Moreover, we have
\begin{align*}
\chi_i \cap \psi_i &= \left\{
\begin{aligned}
& \psi_i & &(i \neq j) \\
& \sigma \cap \psi_j & &(i=j) 
\end{aligned} \right. \\
&=\left\{
\begin{aligned}
& \psi_i & &(i \neq j) \\
& \psi_j \setminus \{x\} & &(i=j) .
\end{aligned} \right.
\end{align*}
Hence, we obtain
\begin{align*}
\chi \cap \psi = \psi \setminus \{x\}.
\end{align*}
Finally, we have $x \in \psi \setminus \phi$ since $x \in \psi_j \setminus \phi_j$.
Therefore, by the above argument, we get $\chi \in \facet{\polyp}$ and $x \in \psi \setminus \phi$ such that
$\chi < \psi$ and  $\chi \cap \psi = \psi \setminus \{x\}$. By Lemma \ref{shellable definition}, we conclude that $\polyp$ is shellable.
\end{proof}

The following claim indicates that we cannot drop condition (2) in Theorem \ref{sufficient} for some $M$.
\begin{claim}
\label{nonshellable}
Let $(K, L)$ be a pair of simplicial complexes such that $\facet{L} \setminus \facet{K} \neq \emptyset$. Then for $M = \langle \{1,2\}, \{3,4\} \rangle$, $\polyp$ is not shellable.
\end{claim}
\begin{proof}
We assume that $\polyp$ is shellable and deduce a contradiction. Take $\tau^0 \in \facet{L} \setminus \facet{K}$. Fix a shelling on $\polyp$ and define $\Phi, \Psi \in \facet{\polyp}$ by
\begin{align*}
\Phi &= \min \left\{ \phi \in \facet{\polyp} \ \middle| \ \phi_1 = \phi_2 =\tau^0, \phi_3 \supsetneq \tau^0, \phi_4 \supsetneq \tau^0 \right\} ,\\
\Psi &= \min \left\{ \psi \in \facet{\polyp} \ \middle| \ \psi_1 \supsetneq \tau^0, \psi_2 \supsetneq \tau^0, \psi_3 = \psi_4 = \tau^0 \right\} .
\end{align*}
Note that $\Phi, \Psi$ are well-defined because there exists $\sigma \in \facet{K} \setminus \facet{L}$ such that $\sigma \supsetneq \tau^0$ since $\tau^0 \in \facet{L} \setminus \facet{K}$. 

We may suppose that $\Phi < \Psi$. Then there exists $\chi < \Psi$, $j \in [m]$ and $x \in \Psi_j \setminus \Phi_j$ such that $\chi \cap \Psi = \Psi \setminus \{x\}$. Since $\Psi_j \setminus \Phi_j \neq \emptyset$, we obtain $j \in \{1,2\}$. Here, we set $j=1$. 
Then for $i =2,3,4$, we have $\chi_i \cap \Psi_i = \Psi_i$, namely $\chi_i \supset \Psi_i$. So, we get $\chi_2 = \Psi_2 \in \facet{K} \setminus \facet{L}$ since $\Psi_2$ is a facet of $K$. Thus, $2 \in \overline{\chi}$. This implies that $\chi_i \in \facet{L}$ for $i=3,4$. Therefore, we obtain $\chi_i = \tau^0$ for $i=3,4$ since $\chi_i \supset \Psi_i = \tau^0$. 

On the other hand, we have $x \in \Psi_1 \setminus \Phi_1 = \Psi_1 \setminus \tau^0$ since $\Phi_1 = \tau^0$. Then it follows from $\tau^0 \subset \Psi_1$ that $\tau^0 \subset \Psi_1 \setminus \{x\} = \chi_1 \cap \Psi_1 \subset \chi_1$. 
If $\chi_1= \tau^0$, then we have $\overline{\chi} = \{2\}$ and $\chi_1 \in \facet{L} \setminus \facet{K}$. This contradicts to Proposition \ref{facet description} since $\chi$ is a facet of $\polyp$. So, we get $\chi_1 \supsetneq \tau^0$. However, this is also a contradiction to the minimality of $\Psi$. Therefore, we conclude that $\polyp$ is not shellable.
\end{proof}

\begin{theorem}
\label{shellable sufficient 2}
Let $M$ be a simplicial complex on $[m]$ and $\pair = \{(K_i, L_i)\}_{i \in [m]}$ be a family of pairs of simplicial complexes. 
Suppose that $M$ is shellable and that for any $i \in [m]$, 
\begin{itemize}
\item there exists $\alpha_i \in \facet{K_i}$ such that $\facet{L_i} \subset \{ \alpha_i \setminus \{x\} \ |\ x \in \alpha_i\}$, and
\item $K_i$ is shellable with a shelling $<_i$ such that $\alpha_i$ is the minimum element with respect to $<_i$.
\end{itemize}
Then $\poly$ is shellable.
\end{theorem}
\begin{proof}
We fix a shelling on $M$ and denote by $<_M$. Take an arbitrary linear order $<'_i$ on $\facet{L_i}$. 
Here we remark that we have $\facet{K_i} \cap \facet{L_i} = \emptyset$ for any $i \in [m]$ since we have $\tau \subsetneq \alpha_i$ for any $\tau \in \facet{L_i}$.

Let $\phi, \psi$ be facets of $\poly$. Consider a linear order $<$ on $\facet{\poly}$ such that $\phi < \psi$ if and only if 
\begin{itemize}
\item $\overline{\phi} <_M \overline{\psi}$, or
\item $\overline{\phi} = \overline{\psi}$ and $\phi_1 = \psi_1, \ldots, \phi_{j-1} = \psi_{j-1}, \phi_j <_j \psi_j$ for some $j \in \overline{\phi}$, or
\item $\overline{\phi} = \overline{\psi}$ and $\phi_1 = \psi_1, \ldots, \phi_{j-1} = \psi_{j-1}, \phi_j <'_j \psi_j$ for some $j \notin \overline{\phi}$.
\end{itemize}
We show that this order is a shelling on $\poly$.

First, suppose that $\overline{\phi} <_M \overline{\psi}$. Then, by Lemma \ref{shellable definition}, there exists $T \in \facet{M}$ and $j \in \overline{\psi} \setminus \overline{\phi}$ such that $T<_M \overline{\psi}$ and $T \cap \overline{\psi} = \overline{\psi} \setminus \{j\}$.
It follows from $j \in \overline{\psi} \setminus \overline{\phi}$ that $\psi_j \in \facet{K_j}$ and $\phi_j \in \facet{L_j}$.

If $\psi_j$ is not minimum with respect to $<_j$, namely $\alpha_j <_j \psi_j$ in $\facet{K_j}$, then there exists $\sigma \in \facet{K_j}$ and $x \in \psi_j \setminus \alpha_j$ such that $\sigma <_j \psi_j$ and $\sigma \cap \psi_j = \psi_j \setminus \{x\}$. By the assumption of the theorem, $\phi_j \in \facet{L_j}$ is a face of $\alpha_j$.
So, we get $x \in \psi_j \setminus \phi_j \subset \psi \setminus \phi$.
Now we define $\chi \subset \bigsqcup_{i \in [m]} V(K_i)$ by
\begin{align*}
\chi_i = \left\{
\begin{aligned}
& \psi_i & &(i \neq j) \\
& \sigma & &(i =j).
\end{aligned} \right.
\end{align*}
Since $\psi_j, \sigma \in \facet{K_j}$, we have $\overline{\chi} = \overline{\psi} \in \facet{M}$. 
Thus, by Proposition \ref{facet description}, $\chi$ is a facet of $\poly$. 
Furthermore, we have $\chi < \psi$ since
\begin{align*}
\chi_1=\psi_1, \ldots, \chi_{j-1} = \psi_{j-1}, \chi_j = \sigma <_j \psi_j.
\end{align*}
Moreover, we have
\begin{align*}
\chi_i \cap \psi_i &= \left\{
\begin{aligned}
& \psi_i & &(i \neq j) \\
& \sigma \cap \psi_j & &(i=j) 
\end{aligned} \right. \\
&=\left\{
\begin{aligned}
& \psi_i & &(i \neq j) \\
& \psi_j \setminus \{x\} & &(i=j) ,
\end{aligned} \right.
\end{align*}
namely
\begin{align*}
\chi \cap \psi = \psi \setminus \{x\}.
\end{align*}
By the above argument, we get $\chi \in \facet{\poly}$ and $x \in \psi \setminus \phi$ such that
$\chi < \psi$ and  $\chi \cap \psi = \psi \setminus \{x\}$. 

If $\psi_j = \alpha_j$, then it follows from the assumption of the theorem that there exists $x \in \alpha_j$ such that $\phi_j = \alpha_j \setminus \{x\}$. In this case, we define $\chi \subset \bigsqcup_{i \in [m]} V(K_i)$ by
\begin{align*}
\chi_i = \left\{
\begin{aligned}
&\psi_i \in \facet{K_i} & &(i \in T \cap \overline{\psi}) \\
&\phi_j \in \facet{L_j} & &(i = j), \\
&\alpha_i \in \facet{K_i} & &(i \in T \setminus \overline{\psi}), \\
&\psi_i \in \facet{L_i} & &(i \in [m] \setminus (T \cup \overline{\psi})). \\
\end{aligned} \right.
\end{align*}
Then we have
\begin{align*}
\overline{\chi} =  (T \cap \overline{\psi}) \cup (T \setminus \overline{\psi}) = T.
\end{align*}
Hence, $\chi$ is a facet of $\poly$. It follows from $T<_M \overline{\psi}$ that $\overline{\chi} <_M \overline{\psi}$. Furthermore, since $\psi_i$ is a face of $\alpha_i$ if $\psi_i \in \facet{L_i}$, we have
\begin{align*}
\chi_i \cap \psi_i &= \left\{
\begin{aligned}
& \psi_i & &(i \in (T \cap \overline{\psi}) \cup ([m] \setminus (T \cup \overline{\phi}))), \\
& \phi_j \cap \psi_j & &(i=j) ,\\
& \alpha_i \cap \psi_i & &( i \in T \setminus \overline{\psi})
\end{aligned} \right. \\
&=\left\{
\begin{aligned}
& \psi_i & &(i \neq j) \\
& \psi_j \setminus \{x\} & &(i=j) .
\end{aligned} \right.
\end{align*}
The last equality follows from
\begin{align*}
\phi_j \cap \psi_j = (\alpha_j \setminus \{x\}) \cap \alpha_j = \alpha_j \setminus \{x\} =\psi_j \setminus \{x\}.
\end{align*}
By the above argument, we get $\chi \in \facet{\poly}$ and $x \in \psi \setminus \phi$ such that $\chi < \psi$ and $\chi \cap \psi = \psi \setminus \{x\}$. 

Next, suppose that $\overline{\phi} = \overline{\psi}$ and $\phi_1 = \psi_1, \ldots, \phi_{j-1} = \psi_{j-1}, \phi_j <_j \psi_j$ for some $j \in \overline{\phi}$. There exists $\sigma \in \facet{K_j}$ and $x \in \psi_j \setminus \phi_j$ such that $\sigma <_j \psi_j$ and $\sigma \cap \psi_j = \psi_j \setminus \{x\}$. We define $\chi \subset \bigsqcup_{i \in [m]} V(K_i)$ by
\begin{align*}
\chi_i = \left\{
\begin{aligned}
& \psi_i & &(i \neq j) \\
& \sigma & &(i =j).
\end{aligned} \right.
\end{align*}
We have $\psi_j \in \facet{K_j}$ since $j \in \overline{\phi} =\overline{\psi}$. So, it follows from $\sigma \in \facet{K_j}$ that $\overline{\chi} = \overline{\psi}$.
Hence, $\chi$ is a facet of $\poly$. 
By the same argument as the first case, we see that $\chi \in \facet{\poly}$ and $x \in \psi \setminus \phi$ satisfy
$\chi < \psi$ and  $\chi \cap \psi = \psi \setminus \{x\}$. 

Finally, suppose that $\overline{\phi} = \overline{\psi}$ and $\phi_1 = \psi_1, \ldots, \phi_{j-1} = \psi_{j-1}, \phi_j <'_j \psi_j$ for some $j \notin \overline{\phi}$. 
By the assumption of the theorem, there exists $x,y \in \alpha_j$ ($x \neq y$) such that $\phi_j = \alpha_j \setminus \{x \}$, $\psi_j = \alpha_j \setminus \{ y \}$. 
So, we have $\phi_j \cap \psi_j = \alpha_j \setminus \{x, y\} = \psi_j \setminus \{x\}$ since $x \neq y$. 
Here we define $\chi \subset \bigsqcup_{i \in [m]} V(K_i)$ by
\begin{align*}
\chi_i = \left\{
\begin{aligned}
& \psi_i & &(i \neq j) \\
& \phi_j & &(i =j).
\end{aligned} \right.
\end{align*}
We have $\phi_j, \psi_j \in \facet{L_j}$ since $j \in \overline{\phi} = \overline{\psi}$. So, it follows that $\overline{\chi} = \overline{\psi}$.
Hence, $\chi$ is a facet of $\poly$. 
By the same argument as the first case, we see that $\chi \in \facet{\poly}$ and $x \in \psi \setminus \phi$ satisfy
$\chi < \psi$ and  $\chi \cap \psi = \psi \setminus \{x\}$. 

In all four cases above, we obtain $\chi \in \facet{\poly}$ and $x \in \psi \setminus \phi$ which satisfy $\chi < \psi$ and  $\chi \cap \psi = \psi \setminus \{x\}$. By Lemma \ref{shellable definition}, we conclude that $\poly$ is shellable.
\end{proof}

\begin{remark}
\label{sufficient2 remark}
Consider $M = \langle \{1\}, \{2\} \rangle$ and $(\underline{K}, \underline{L}) = \{(K_i,L_i)\}_{i=1,2}$ such that
\begin{align*}
K_i =\langle \{a_i,b_i\}, \{b_i, c_i\} \rangle , \ L_i =\langle \{a_i\}, \{c_i\} \rangle .
\end{align*}
Then we see that
\begin{align*}
\poly = &\langle \{a_1,b_1,a_2\}, \{a_1,b_1,c_2\}, \{b_1,c_1, a_2\}, \{b_1,c_1,c_2\},\\
&\ \{a_1, a_2, b_2\}, \{a_1, b_2, c_2\}, \{c_1, a_2, b_2\}, \{c_1, b_2, c_2\} \rangle
\end{align*}
(see Figure \ref{sufficient2 remark figure}) is shellable. However, there is no $\alpha_i \in \{ \{a_i,b_i\}, \{b_i, c_i\} \}$ such that $\{a_i\}, \{c_i\} \subset \alpha_i$. So, the condition that $\facet{L_i} \subset \{ \alpha_i \setminus \{x\} \in K \ |\ x \in \alpha_i\}$ for some $\alpha_i \in \facet{K_i}$ is not necessary for $\poly$ being shellable.

\begin{figure}[tb]
\begin{tabular}{ccccc}
\begin{tikzpicture}[scale=0.8]
\draw (1,1)--(1,3);
\node at (1.5,2) {$*$}; 
\node at (1,1) [vertex] {}; \node at (1,1) [left] {$c_1$}; 
\node at (1,2) [vertex] {}; \node at (1,2) [left] {$b_1$};
\node at (1,3) [vertex] {}; \node at (1,3) [left] {$a_1$};
\node at (2,1) [vertex] {}; \node at (2,1) [right] {$c_2$};
\node at (2,3) [vertex] {}; \node at (2,3) [right] {$a_2$};
\end{tikzpicture}
& &
\begin{tikzpicture}[scale=0.8]
\node at (1,1) {}; \node at (1,2) {$\cup$}; \node at (1,3) {};
\end{tikzpicture}
& &
\begin{tikzpicture}[scale=0.8]
\draw (2,1)--(2,3);
\node at (1.5,2) {$*$};
\node at (1,1) [vertex] {}; \node at (1,1) [left] {$c_1$};
\node at (2,2) [vertex] {}; \node at (2,2) [right] {$b_2$};
\node at (1,3) [vertex] {}; \node at (1,3) [left] {$a_1$};
\node at (2,1) [vertex] {}; \node at (2,1) [right] {$c_2$};
\node at (2,3) [vertex] {}; \node at (2,3) [right] {$a_2$};
\end{tikzpicture}
\end{tabular}
\caption{$\poly$ in Remark \ref{sufficient2 remark}}
\label{sufficient2 remark figure}
\end{figure}
\end{remark}

\begin{example}
Let $K$ be a simplicial complex on $[m]$ and $J=(j_1, \ldots, j_m)$ be a tuple of positive integers. Define families of pairs of simplicial complexes as follows:
\begin{align*}
(\underline{\Delta^{J-1}}, \underline{\partial \Delta^{J-1}} ) &=\{( \Delta^{j_i-1}, \partial \Delta^{j_i-1} )\}_{i \in [m]}, \\
(\underline{\mathrm{pt}^J}, \{\emptyset\}) &=\left\{ \left( {\bigsqcup}_{j_i} \mathrm{pt}, \{\emptyset\} \right) \right\}_{i \in [m]}.
\end{align*}
By Theorem \ref{shellable sufficient 2}, $\mathcal{Z}^*_K (\underline{\Delta^{J-1}}, \underline{\partial {\Delta^{J-1}}})$ and $\mathcal{Z}^*_K (\underline{\mathrm{pt}^{J}}, \{\emptyset\})$ are shellable if $K$ is shellable.
\end{example}

Next, we state a necessary condition for $\polyp$ being shellable under a certain assumption.
\begin{theorem}
\label{necessary1}
Let $M$ be a simplicial complex which is not a simplex and $K \supsetneq L$ be a pair of simplicial complexes such that $\facet{L} \subset \facet{K}$. If $\polyp$ is shellable, then both $K$ and $L$ are shellable.
\end{theorem}
\begin{proof}
We set $V(M)=[m]$. Let $\pair = \{(K_i, L_i)\}_{i \in [m]}$ be a family of pairs of simplicial complexes such that $K_i=K$, $L_i=L$ for any $i \in [m]$. 
Fix a shelling $<$ on $\poly$. Take $S \in \facet{M}$ and $k \in [m]$ such that $k \notin S$. This is possible since $M$ is not a simplex. We prove that both $K_k$ and $L_k$ are shellable.

First, we prove that $K_k$ is shellable. For a facet $\sigma \in \facet{K_k}$, define $\Phi^\sigma \in \facet{\poly}$ by
\begin{align*}
\Phi^\sigma = \min \left\{\phi \in \facet{\poly} \ \middle|\ \phi_k = \sigma \right\}.
\end{align*}
Note that $\Phi^\sigma$ is well-defined since by Remark \ref{facet description remark}, there exists a facet $\phi \in \facet{\poly}$ such that
\begin{align*}
\phi_i = \left\{
\begin{aligned}
& \sigma & &(i =k) \\
& \tau & &(i \neq k),
\end{aligned} \right.
\end{align*}
where $\tau$ is an arbitrary facet of $L$. 
For $\sigma, \sigma' \in \facet{K_k}$, we define a relation $\sigma <_K \sigma'$ by $\Phi^\sigma < \Phi^{\sigma'}$. It is obvious that $<_K$ defines a linear order on $\facet{K_k}$. We prove that $<_K$ is a shelling on $K_k$.

For $\sigma', \sigma \in \facet{K_k}$ such that $\Phi^{\sigma'} < \Phi^\sigma$, there exists $\chi \in \facet{\poly}$ and $x \in \Phi^\sigma \setminus \Phi^{\sigma'}$ such that $\chi < \Phi^\sigma$ and $\chi \cap \Phi^\sigma = \Phi^\sigma \setminus \{x\}$. If $x \notin \Phi^\sigma_k$, then we have $\chi_k \cap \Phi^\sigma_k = \Phi^\sigma_k$, which means that $\Phi^\sigma_k \subset \chi_k$. By Proposition \ref{facet description}, $\Phi^\sigma_k$ and $\chi_k$ are facets of $K_k$ since $\facet{K_k} \cup \facet{L_k} = \facet{K_k}$. So, we obtain $\chi_k = \Phi^\sigma_k = \sigma$. This contradicts to the minimality of $\Phi^\sigma$. Hence, we get $x \in \Phi^\sigma_k$. 
Moreover, it follows from $x \in \Phi^\sigma \setminus \Phi^{\sigma'}$ that $x \in \Phi^\sigma_k \setminus \Phi^{\sigma'}_k = \sigma \setminus \sigma'$. 

By $\chi \cap \Phi^\sigma = \Phi^\sigma \setminus \{x\}$ and $x \in \Phi^\sigma_k$, we obtain $\chi_k \cap \sigma = \chi_k \cap \Phi^\sigma_k = \Phi^\sigma_k \setminus \{x\} = \sigma \setminus \{x\}$. 
Furthermore, we get $\Phi^{\chi_k} <\chi < \Phi^\sigma$. 
Therefore, $\chi_k \in \facet{K_k}$ and $x \in \sigma \setminus \sigma'$ satisfy $\chi_k <_K \sigma$ and $\chi_k \cap \sigma = \sigma \setminus \{x\}$. By Lemma \ref{shellable definition}, $<_K$ is a shelling on $K_k$.

Next, we prove that $L_k$ is shellable. There exists $\sigma^0 \in \facet{K_k} \setminus \facet{L_k}$ since $K \neq L$. For a facet $\tau \in \facet{L_k}$, define $\Psi^\tau \in \facet{\poly}$ by
\begin{align*}
\Psi^\tau = \min \left\{\psi \in \facet{\poly} \ \middle|\ \psi_i = \sigma^0 \text{ for any } i \in S, \text{ and } \psi_k=\tau \right\}.
\end{align*}
Note that $\Psi^\tau$ is well-defined. This is because there exists a facet $\psi \in \facet{\poly}$ defined by
\begin{align*}
\psi_i = \left\{
\begin{aligned}
& \sigma^0 & &(i \in S) \\
& \tau & &(i \notin S).
\end{aligned} \right.
\end{align*}
For $\tau, \tau' \in \facet{L_k}$, we define a relation $\tau <_L \tau'$ by $\Psi^\tau < \Psi^{\tau'}$. It is obvious that $<_L$ defines a linear order on $\facet{L_k}$. We prove that $<_L$ is a shelling on $L_k$.

For $\tau', \tau \in \facet{L_k}$ such that $\Psi^{\tau'} < \Psi^\tau$, there exists $\chi \in \facet{\poly}$ and $x \in \Psi^\tau \setminus \Psi^{\tau'}$ such that $\chi < \Psi^\tau$ and $\chi \cap \Psi^\tau = \Psi^\tau \setminus \{x\}$. Let $j \in [m]$ be a vertex of $M$ such that $x \in \Psi^\tau_j \setminus \Psi^{\tau'}_j$. Then we have $j \notin S$ since $\Psi^\tau_j \neq \Psi^{\tau'}_j$. 
So, for any $i \in S$, we get $\chi_i \cap \Psi^\tau_i = \Psi^\tau_i$, namely $\chi_i \supset \Psi^\tau_i = \sigma^0$. It follows from $\sigma^0 \in \facet{K_k}$ that $\chi_i = \sigma^0$ for any $i \in S$. Hence, by the minimality of $\Psi^\tau$, we get $\chi_k \neq \tau$.
On the other hand, we get $\overline{\chi} =S$ since $S$ is a facet of $M$ and $S \subset \overline{\chi}$. So, we have $k \notin S = \overline{\chi}$, which means that $\chi_k \in \facet{L_k}$. Therefore, if $j \neq k$, we obtain $\chi_k \supsetneq \tau$ and $\chi_k, \tau \in \facet{L_k}$, a contradiction. 
Thus, we conclude that $j=k$. Hence, we get $x \in \Psi^\tau_k \setminus \Psi^{\tau'}_k$ and $\chi_k \cap \Psi^\tau_k = \Psi^\tau_k \setminus \{x\}$, namely $x \in \tau \setminus \tau'$ and $\chi_k \cap \tau = \tau \setminus \{x\}$. 
Moreover, we have $\Psi^{\chi_k} <\chi$ since $\chi_i = \sigma^0$ for any $i \in S$ and $\chi_k \in \facet{L_k}$. Hence, by $\chi < \Psi^\tau$, we obtain $\Psi^{\chi_k} < \Psi^{\tau}$, namely $\chi_k <_L \tau$. By Lemma \ref{shellable definition}, $<_L$ is a shelling on $L_k$.
\end{proof}

\begin{remark}
\label{necessary1 remark}
In general, the shellability of $\poly$ does not imply the shellability of each $K_i$. For example, consider $M = \langle \{1\}, \{2\} \rangle$ and $(\underline{K}, \underline{L}) = \{(K_i,L_i) \}_{i=1,2}$ such that
\begin{alignat*}{3}
&K_1 =\langle \{a,b\},\{c,d \} \rangle , &\  &L_1 =\langle \{b\}, \{c, d\} \rangle ,\\
&K_2 =\langle \{e,f\} \rangle , &\  & L_2 =\langle \{f\} \rangle.
\end{alignat*}
Then we see that
\begin{align*}
\poly = \langle \{c,d,e,f\}, \{b,e,f\}, \{a,b,f\} \rangle
\end{align*}
(see Figure \ref{necessary1 remark figure}) is shellable. However, $K_1$ is not shellable.

\begin{figure}[tb]
\begin{tabular}{ccccc}
\begin{tikzpicture}[scale=0.8]
\draw (1,1)--(1,2) (1,3)--(1,4);
\node at (1.5,2.5) {$*$};
\node at (1,1) [vertex] {}; \node at (1,1) [left] {$d$};
\node at (1,2) [vertex] {}; \node at (1,2) [left] {$c$};
\node at (1,3) [vertex] {}; \node at (1,3) [left] {$b$};
\node at (1,4) [vertex] {}; \node at (1,4) [left] {$a$};
\node at (2,2) [vertex] {}; \node at (2,2) [right] {$f$};
\end{tikzpicture}
& &
\begin{tikzpicture}[scale=0.8]
\node at (1,1) {}; \node at (1,2.5) {$\cup$}; \node at (1,4) {};
\end{tikzpicture}
& &
\begin{tikzpicture}[scale=0.8]
\draw (2,2)--(2,3) (1,1)--(1,2);
\node at (1.5,2.5) {$*$};
\node at (1,1) [vertex] {}; \node at (1,1) [left] {$d$};
\node at (1,2) [vertex] {}; \node at (1,2) [left] {$c$};
\node at (1,3) [vertex] {}; \node at (1,3) [left] {$b$};
\node at (1,4) {};
\node at (2,2) [vertex] {}; \node at (2,2) [right] {$f$};
\node at (2,3) [vertex] {}; \node at (2,3) [right] {$e$};
\end{tikzpicture}
\end{tabular}
\caption{$\poly$ in Remark \ref{necessary1 remark}}
\label{necessary1 remark figure}
\end{figure}
\end{remark}

\begin{remark}
The shellability of $\polyp$ does not imply $\facet{L} \subset \facet{K}$. 
For example, consider $M = \langle \{1\}, \{2\} \rangle$ and $(\underline{K}, \underline{L}) = \{(K_i,L_i)\}_{i=1,2}$ such that
\begin{align*}
K_i =\langle \{a_i,b_i\} \rangle , \ L_i =\langle \{b_i\} \rangle .
\end{align*}
Then we see that
\begin{align*}
\poly = \langle \{a_1,b_1,b_2\}, \{b_1,a_2, b_2\} \rangle
\end{align*}
is shellable. However, $\{b_i\} \in \facet{L_i} \setminus \facet{K_i}$ for $i=1,2$.  
\end{remark}

Even though $(K, L)$ satisfies that $K$ is shellable, $L$ is shellable and $\facet{L} \subset \facet{K}$, $\polyp$ is not necessarily shellable.
In order to show this, we prove another necessary condition.
\begin{theorem}
\label{necessary2}
Let $M$ be a simplicial complex which is not a simplex and $K \supsetneq L$ be a pair of simplicial complexes. Suppose that $\polyp$ is shellable. Then there exists a pair $(\sigma, \tau)$ of $\sigma \in \facet{K} \setminus \facet{L}$ and $\tau \in \facet{L}$ such that 
\begin{align*}
|\sigma \cap \tau| = \max_{\rho \in \facet{K} \setminus \facet{L}} |\rho| -1 .
\end{align*}
\end{theorem}
\begin{proof}
Fix a shelling $\phi^1< \cdots <\phi^t$ on $\polyp$ such that $|\phi^p| \geq |\phi^q|$ for any $1 \leq p < q \leq t$, which exists by Lemma \ref{rearrangement lemma}. 

We first show that there exists an index $s$ such that $\overline{\phi^p} \nsupseteq \overline{\phi^s}$ for any $p \in [s-1]$. Since $M$ is not a simplex, there exists $j \in [m]$ such that $j \notin \overline{\phi^1}$. Take a facet $S \ni j$ of $M$. Here, it follows from $K \neq L$ that there exists a facet $\rho \in \facet{K} \setminus \facet{L}$ of $K$. Now, consider a subset $\psi \subset \bigsqcup_{i \in [m]} V(K)$ which satisfies
\begin{itemize}
\item if $i \in S$, then $\psi_i = \rho$,
\item if $i \notin S$, then $\psi_i \in \facet{L}$.
\end{itemize}
We have $\psi_i \in \facet{K} \cup \facet{L}$ for any $i \in [m]$, $\overline{\psi}= S \in M$ and $\overline{\psi} \cup \{i\} \notin M$ for any $i \notin S$ since $S$ is a facet of $M$. By Proposition \ref{facet description}, $\psi$ is a facet of $\polyp$. Therefore, we have $\psi= \phi^q$ for some $q \in [t]$. Since $j \in S \setminus \overline{\phi^1} = \overline{\phi^q} \setminus \overline{\phi^1}$, we get $\overline{\phi^q} \nsubseteq \overline{\phi^1}$. 
Now we set $s = \min \{r \in [t] \ |\ \overline{\phi^r} \nsubseteq \overline{\phi^1} \}$. It follows from the minimality of $s$ that $\overline{\phi^p} \subset \overline{\phi^1}$ for any $p \in [s-1]$. Thus, we get $\overline{\phi^s} \nsubseteq \overline{\phi^p}$ for any $p \in [s-1]$ as desired.

Now, take an index $s$ as above. 
For an index $p'$ such that $\overline{\phi^{p'}} = \overline{\phi^s}$, $p'$ must be larger than $s$, which implies $|\phi^{p'}| \leq |\phi^s|$. So, $\phi^s$ is the largest facet among facets $\phi$ such that $\overline{\phi} = \overline{\phi^s}$. 
By the definition of a shelling, there must exist $p$ with $1 \leq  p <s$ such that $|\phi^p \cap \phi^s| = |\phi^s| -1$. This equality implies that we have 
\begin{align*}
|\phi^p_i \cap \phi^s_i | \geq |\phi^s_i| -1
\end{align*}
for any $i \in [m]$.
On the other hand, it follows from $\overline{\phi^p} \nsupseteq \overline{\phi^s}$ that there exists $j \in \overline{\phi^s} \setminus \overline{\phi^p}$, namely $j \in [m]$ such that $\phi^p_j \in \facet{L}$ and $\phi^s_j \in \facet{K} \setminus \facet{L}$. 
Since $\phi^s_j$ is a facet of $K$, we have $\phi^p_j \cap \phi^s_j \subsetneq \phi^s_j$. Thus, we get
\begin{align*}
|\phi^p_j \cap \phi^s_j | = |\phi^s_j| -1.
\end{align*}
Assume that there exists $\mu \in \facet{K} \setminus \facet{L}$ such that $|\mu| > |\phi^s_j|$. Then, a subset $\chi \subset \bigsqcup_{i \in [m]} V(K)$ defined by 
\begin{align*}
\chi_i =\left\{
\begin{aligned}
& \phi^s_i & &(i \neq j) \\
& \mu & &(i =j)
\end{aligned} \right.
\end{align*}
is a facet of $\polyp$. Since $|\chi| > |\phi^s|$ and $\overline{\chi} = \overline{\phi^s}$, this contradicts to the maximality of $|\phi^s|$. Hence, we obtain $|\phi^s_j| = \max_{\rho \in \facet{K} \setminus \facet{L}} |\rho|$.

Then, $\sigma = \phi^s_j \in \facet{K} \setminus \facet{L}$ and $\tau = \phi^p_j \in \facet{L}$ satisfy
\begin{align*}
|\sigma \cap \tau | = |\sigma| -1 = \max_{\rho \in \facet{K} \setminus \facet{L}} |\rho| -1,
\end{align*}
which is the desired conclusion.
\end{proof}

\begin{example}
\label{necessary2 example}
Consider $M= \langle \{1\}, \{2\} \rangle$ and $(\underline{K}, \underline{L}) = \{(K_i,L_i)\}_{i=1,2}$ such that
\begin{align*}
K_i =\langle \{a_i,b_i\}, \{b_i, c_i\}, \{d_i\} \rangle , \ L_i =\langle \{c_i\}, \{d_i\} \rangle .
\end{align*}
For $\{a_i,b_i\} \in \facet{K_i} \setminus \facet{L_i}$, there exists no $\tau \in \facet{L_i}$ such that $|\{a_i,b_i\} \cap \tau| =1$. For $\{d_i\} \in \facet{L_i}$, there exists no $\sigma \in \facet{K_i} \setminus \facet{L_i}$ such that $|\sigma \cap \{d_i\}|=1$. On the other hand, we see that
\begin{align*}
\poly = &\langle \{a_1, b_1, c_2\}, \{b_1, c_1, c_2\}, \{b_1, c_1, d_2\}, \{a_1, b_1, d_2\}, \\
&\ \{c_1, b_2, c_2\}, \{d_1, b_2, c_2\}, \{d_1, a_2, b_2\}, \{c_1, a_2, b_2\}, \{d_1, d_2\} \rangle
\end{align*}
(see Figure \ref{necessary2 example figure}) is shellable. So, Theorem \ref{necessary2} only guarantees the existence of at least one pair $(\sigma, \tau)$, in this example, $\sigma=\{b_i, c_i \}$ and $\tau = \{c_i\}$.

\begin{figure}[tb]
\begin{tabular}{ccccc}
\begin{tikzpicture}[scale=0.8]
\draw (1,2)--(1,4);
\node at (1.5,2.5) {$*$};
\node at (1,1) [vertex] {}; \node at (1,1) [left] {$d_1$};
\node at (1,2) [vertex] {}; \node at (1,2) [left] {$c_1$};
\node at (1,3) [vertex] {}; \node at (1,3) [left] {$b_1$};
\node at (1,4) [vertex] {}; \node at (1,4) [left] {$a_1$};
\node at (2,2) [vertex] {}; \node at (2,2) [right] {$c_2$};
\node at (2,1) [vertex] {}; \node at (2,1) [right] {$d_2$};
\end{tikzpicture}
& &
\begin{tikzpicture}[scale=0.8]
\node at (1,1) {}; \node at (1,2.5) {$\cup$}; \node at (1,4) {};
\end{tikzpicture}
& &
\begin{tikzpicture}[scale=0.8]
\draw (2,2)--(2,4);
\node at (1.5,2.5) {$*$};
\node at (1,1) [vertex] {}; \node at (1,1) [left] {$d_1$};
\node at (1,2) [vertex] {}; \node at (1,2) [left] {$c_1$};
\node at (2,3) [vertex] {}; \node at (2,3) [right] {$b_2$};
\node at (2,4) [vertex] {}; \node at (2,4) [right] {$a_2$};
\node at (2,2) [vertex] {}; \node at (2,2) [right] {$c_2$};
\node at (2,1) [vertex] {}; \node at (2,1) [right] {$d_2$};
\end{tikzpicture}
\end{tabular}
\caption{$\poly$ in Example \ref{necessary2 example}}
\label{necessary2 example figure}
\end{figure}
\end{example}

\begin{remark}
\label{necessary2 remark}
The converse of Theorem \ref{necessary1} and the converse of Theorem \ref{necessary2} do not hold. Consider $M= \langle \{1\}, \{2\} \rangle$ and $(\underline{K}, \underline{L}) = \{(K_i,L_i)\}_{i=1,2}$ such that
\begin{align*}
K_i =\langle \{a_i, b_i, c_i \}, \{a_i ,c_i ,d_i\}, \{b_i, d_i \} \rangle , \ L_i =\langle \{a_i, b_i, c_i \}, \{b_i, d_i \} \rangle .
\end{align*}
Two facets $\{a_i, c_i, d_i\} \in \facet{K_i} \setminus \facet{L_i}$ and $\{a_i,b_i,c_i\} \in \facet{L_i}$ satisfy the condition in Theorem \ref{necessary2}. Moreover, both $K=K_1=K_2$ and $L=L_1=L_2$ are shellable and $\facet{L} \subset \facet{K}$. However, 
\begin{align*}
\poly = &\langle \{a_1, b_1, c_1, a_2, b_2, c_2\}, \{a_1, c_1, d_1,a_2, b_2, c_2\}, \\
&\ \{a_1, b_1, c_1, a_2, c_2, d_2\}, \{a_1, b_1, c_1, b_2, d_2\}, \{a_1, c_1, d_1, b_2, d_2\}, \\
&\ \{b_1, d_1, a_2, b_2, c_2\}, \{b_1, d_1, a_2, c_2, d_2\}, \{b_1, d_1, b_2, d_2\} \rangle
\end{align*}
(see Figure \ref{necessary2 remark figure}) is not shellable. To see this, assume that $\poly$ is shellable and $\{a_1, c_1, d_1, b_2, d_2\} < \{b_1, d_1, a_2, c_2, d_2\}$ in a shelling. There must exists a facet $\phi \neq \{b_1, d_1, a_2, c_2, d_2\}$ of $\poly$ and $x \in \{b_1, d_1, a_2, c_2, d_2\} \setminus \{a_1, c_1, d_1, b_2, d_2\} = \{b_1, a_2, c_2\}$ such that 
$\{b_1, d_1, a_2, c_2, d_2\} \setminus \{x\} \subset \phi$. 
However, there is no facet $\phi \neq \{b_1, d_1, a_2, c_2, d_2\}$ such that $\phi$ includes $\{ d_1, a_2, c_2, d_2\}, \{b_1, d_1, c_2, d_2\}$ or $\{b_1, d_1, a_2, d_2\}$. This is a contradiction. In the same way, we also get $ \{b_1, d_1, a_2, c_2, d_2\} \nleq \{a_1, c_1, d_1, b_2, d_2\}$. So, $\poly$ has no shelling.

\begin{figure}[tb]
\begin{tabular}{ccc}
\begin{tikzpicture}[scale=0.8]
\node at (2,2) [vertex] {}; \node at (2,2) [below] {$c_1$};
\node at (2,3) [vertex] {}; \node at (2,3) [above] {$a_1$};
\node at (1.2, 1.5) [vertex] {}; \node at (1.2,1.5) [left] {$b_1$};
\node at (2.8,1.5) [vertex] {}; \node at (2.8,1.5) [right] {$d_1$};
\draw (2,2)--(2,3)--(1.2,1.5)--cycle (2,2)--(2,3)--(2.8,1.5)--cycle (1.2,1.5)--(2.8,1.5);
\fill [opacity=0.3] (2,2)--(2,3)--(1.2,1.5)--cycle (2,2)--(2,3)--(2.8,1.5)--cycle;
\node at (3.5,2.25) {$*$};
\node at (5,2) [vertex] {}; \node at (5,2) [below] {$c_2$};
\node at (5,3) [vertex] {}; \node at (5,3) [above] {$a_2$};
\node at (4.2, 1.5) [vertex] {}; \node at (4.2,1.5) [left] {$b_2$};
\node at (5.8,1.5) [vertex] {}; \node at (5.8,1.5) [right] {$d_2$};
\draw (5,2)--(5,3)--(4.2,1.5)--cycle (4.2,1.5)--(5.8,1.5);
\fill [opacity=0.3] (5,2)--(5,3)--(4.2,1.5)--cycle;
\end{tikzpicture}
&
\begin{tikzpicture}[scale=0.8]
\node at (1,1.5) {}; \node at (1,2.25) {$\cup$}; \node at (1,3) {};
\end{tikzpicture}
&
\begin{tikzpicture}[scale=0.8]
\node at (2,2) [vertex] {}; \node at (2,2) [below] {$c_1$};
\node at (2,3) [vertex] {}; \node at (2,3) [above] {$a_1$};
\node at (1.2, 1.5) [vertex] {}; \node at (1.2,1.5) [left] {$b_1$};
\node at (2.8,1.5) [vertex] {}; \node at (2.8,1.5) [right] {$d_1$};
\draw (2,2)--(2,3)--(1.2,1.5)--cycle (1.2,1.5)--(2.8,1.5);
\fill [opacity=0.3] (2,2)--(2,3)--(1.2,1.5)--cycle ;
\node at (3.5,2.25) {$*$};
\node at (5,2) [vertex] {}; \node at (5,2) [below] {$c_2$};
\node at (5,3) [vertex] {}; \node at (5,3) [above] {$a_2$};
\node at (4.2, 1.5) [vertex] {}; \node at (4.2,1.5) [left] {$b_2$};
\node at (5.8,1.5) [vertex] {}; \node at (5.8,1.5) [right] {$d_2$};
\draw (5,2)--(5,3)--(4.2,1.5)--cycle (5,2)--(5,3)--(5.8,1.5)--cycle (4.2,1.5)--(5.8,1.5);
\fill [opacity=0.3] (5,2)--(5,3)--(4.2,1.5)--cycle (5,2)--(5,3)--(5.8,1.5)--cycle;
\end{tikzpicture}
\end{tabular}
\caption{$\poly$ in Remark \ref{necessary2 remark}}
\label{necessary2 remark figure}
\end{figure}
\end{remark}

\begin{corollary}
\label{dimension corollary}
Let $M$ be a simplicial complex and $K \supset L$ be a pair of simplicial complexes. If $\dim K - \dim L \geq 2$, then $\polyp$ is not shellable. 

In particular, for any simplicial complex $M$ and $K$ such that $\dim K \geq 1$, $\mathcal{Z}^*_M (K, \{\emptyset\})$ is not shellable.
\end{corollary}
\begin{proof}
It follows from $\dim K > \dim L$ that there must be $\rho \in \facet{K}$ such that $|\rho| = \dim K +1$ and $\rho \notin \facet{L}$. 
So, we get
\begin{align*}
\dim K = \max_{\rho \in \facet{K} \setminus \facet{L}} |\rho| -1 .
\end{align*}

For any $\sigma \in \facet{K} \setminus \facet{L}$ and $\tau \in \facet{L}$, we have
\begin{align*}
|\sigma \cap \tau| \leq |\tau| \leq \dim L +1 \leq \dim K -1 = \max_{\rho \in \facet{K} \setminus \facet{L}} |\rho| -2 .
\end{align*}
Therefore, there is no pair $(\sigma, \tau)$ of $\sigma \in \facet{K} \setminus \facet{L}$ and $\tau \in \facet{L}$ such that 
\begin{align*}
|\sigma \cap \tau| = \max_{\rho \in \facet{K} \setminus \facet{L}} |\rho| -1 .
\end{align*}
By Theorem \ref{necessary2}, we conclude that $\polyp$ is not shellable.
\end{proof}

As the last result in this section, we prove that under a certain condition, the shellability of both $K$ and $L$ is equivalent to the shellability of $\polyp$. 
\begin{theorem}
\label{equivalent}
Let $M$ be a simplicial complex which is not a simplex and $K$ be a simplicial complex. For a vertex $v_0$ of $K$, suppose that $\facet{\dl{K}{v_0}} \subset \facet{K}$. Then $\mathcal{Z}^*_M (K, \dl{K}{v_0})$ is shellable if and only if both $K$ and $\dl{K}{v_0}$ are shellable.
\end{theorem}
\begin{proof}
We set $L = \dl{K}{v_0}$.
If $\mathcal{Z}^*_M (K, L)$ is shellable and $\facet{L} \subset \facet{K}$, then by Theorem \ref{necessary1}, $K$ and $L$ are shellable. In order to prove the converse, it is sufficient to show that there is a shelling $<_K$ on $K$ such that $\tau <_K \sigma$ for any $\tau \in \facet{L}$ and any $\sigma \in \facet{K} \setminus \facet{L}$. Then the proof is completed by Theorem \ref{sufficient}.

Let $<$ be a shelling on $K$ and $<'$ be a shelling on $L$. Define a relation $<_K$ on $\facet{K}$ by
\begin{itemize}
\item $\tau <_K \sigma$ for any $\tau \in \facet{L}$ and $\sigma \in \facet{K} \setminus \facet{L}$,
\item for any $\tau, \tau' \in \facet{L}$, $\tau <_K \tau'$ if and only if $\tau <' \tau'$, and
\item for any $\sigma, \sigma' \in \facet{K} \setminus \facet{L}$, $\sigma <_K \sigma'$ if and only if $\sigma < \sigma'$.
\end{itemize}
It is obvious that $<_K$ is a linear order. We prove that $<_K$ is a shelling on $K$. The goal of the proof is to show that for any $\rho, \rho' \in \facet{K}$ such that $\rho <_K \rho'$, there exists $\rho'' \in \facet{K}$ and $u \in \rho' \setminus \rho$ which satisfy $\rho'' <_K \rho'$ and $\rho'' \cap \rho' = \rho' \setminus \{u\}$.

For any $\tau, \tau' \in \facet{L}$ such that $\tau <' \tau'$, there exists $\tau'' \in \facet{L}$ and $x \in \tau' \setminus \tau$ such that $\tau'' <' \tau'$ and $\tau'' \cap \tau' = \tau' \setminus \{x\}$. Since $\tau'' \in \facet{L}$ and $\tau'' <' \tau'$, we get $\tau'' <_K \tau'$.

For any $\sigma, \sigma' \in \facet{K} \setminus \facet{L}$ such that $\sigma < \sigma'$, there exists $\rho \in \facet{K}$ and $x \in \sigma' \setminus \sigma$ such that $\rho < \sigma'$ and $\rho \cap \sigma' = \sigma' \setminus \{x\}$. If $\rho \in \facet{L}$, then we obtain $\rho <_K \sigma'$ since $\sigma' \in \facet{K} \setminus \facet{L}$. If $\rho \in \facet{K} \setminus \facet{L}$, then we obtain $\rho <_K \sigma'$ again since $\rho < \sigma'$. In both cases, we have $\rho <_K \sigma'$.

For $\tau \in \facet{L}$ and $\sigma \in \facet{K} \setminus \facet{L}$, we have $v_0 \in \sigma \setminus \tau$ since $\sigma \notin \dl{K}{v_0}$ and $\tau \in \dl{K}{v_0}$. A simplex $\sigma \setminus \{v_0\} \in L$ is not a facet of $L$ since $\sigma \setminus \{v_0\}$ is not a facet of $K$ and we have $\facet{L} \subset \facet{K}$ by the assumption. Therefore, there exists a facet $\rho \in \facet{L}$ such that $\rho \supsetneq \sigma \setminus \{v_0\}$. Since $\rho \in \dl{K}{v_0}$, we obtain $\rho \cap \sigma = \sigma \setminus \{v_0\}$.
We also get $\rho <_K \sigma$ because $\rho \in \facet{L}$ and $\sigma \in \facet{K} \setminus \facet{L}$.
\end{proof}

\section{Applications to the shellability of graphs}
\label{applications}
In this section, we prove Theorem \ref{main theorem}. Here we introduce a class of graphs which are constructed from two graphs $G, H$ and a subset $U \subset V(H)$. 
In the following, a set $\{u,v\} \subset V(G)$ of two vertices of a graph $G$ is denoted by $uv$.

\begin{definition}
Let $G, H$ be graphs and $U \subset V(H)$ be a subset of $V(H)$. We define a graph $G[H;U]$ by
\begin{align*}
V(G[H;U]) &= V(G) \times V(H) , \\
E(G[H;U]) &= \left\{(u_1,v_1)(u_2,v_2) \ \middle|\ \left.
\begin{aligned}
&u_1 = u_2 \text{ and } v_1 v_2 \in E(H) ,\\
&\text{ or } \\
&u_1 u_2 \in E(G) \text{ and } v_1, v_2 \in U
\end{aligned} \right. \right\}.
\end{align*}
\end{definition}

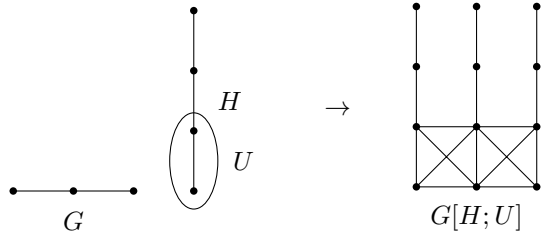
\begin{figure}[tb]
\label{definition figure}
\begin{tabular}{ccccc}
\begin{tikzpicture}[scale=0.8]
\node at (2,0.5) {$G$};
\node at (4.6,2.5) {$H$};
\foreach \x in {1,2,3}  {\node at (\x, 1) [vertex] {};};
\foreach \y in {1,2,3,4} {\node at (4,\y) [vertex] {};};
\draw (1,1)--(3,1) (4,1)--(4,4);
\draw (4,1.5) circle (0.4cm and 0.8cm);
\node at (4.5,1.5) [right] {$U$};
\end{tikzpicture}
& &
\begin{tikzpicture}[scale=0.8]
\node at (1,0) {}; \node at (1,2) {$\rightarrow$}; \node at (1,4) {};
\end{tikzpicture}
& &
\begin{tikzpicture}[scale=0.8]
\foreach \x in {1,2,3} { \foreach \y in {1,2,3,4} {\node at (\x, \y) [vertex] {};};};
\draw (1,1)--(1,4) (2,1)--(2,4) (3,1)--(3,4) (1,1)--(3,1) (1,2)--(3,2) (1,1)--(2,2)--(3,1) (1,2)--(2,1)--(3,2);
\node at (2,0.5) {$G[H;U]$};
\end{tikzpicture}
\end{tabular}
\caption{An example of $G[H;U]$}
\end{figure}

\begin{example}
\label{lexicographic product}
For graphs $G$ and $H$, $G[H; V(H)]$ is the {\it lexicographic product} $G[H]$. The definition of the lexicographic product of two graphs is given in, for example, \cite{Harary69}.
\end{example}

\begin{example}
Let $G, H$ be graphs and $v_0$ be a vertex of $H$. As defined in Section \ref{introduction}, $G[H; \{v_0\}]$ is the graph obtained from $G \sqcup \left( \bigsqcup_{u \in V(G)} H_u \right)$ by identifying $u \in V(G)$ with $v_0 \in H_u$, where $H_u$ is a copy of $H$.
\end{example}

The independence complex of $G[H;U]$ is described as a polyhedral join, stated in the following proposition.
\begin{proposition}
\label{independence complex}
Let $G, H$ be graphs and $U \subset V(H)$ be a subset of $V(H)$. Then we have
\begin{align*}
I(G[H;U]) = \polyind,
\end{align*}
where $H \setminus U$ is a graph defined by $V(H \setminus U)=V(H) \setminus U$ and $E(H \setminus U)=\{uv \in E(H) \ |\ u, v \notin U\}$.
\end{proposition}
\begin{proof}
Remark that for $\phi \subset V(G) \times V(H) = V(\polyind)$, $\phi_u$ ($u \in V(G)$) and $\overline{\phi}$ in Proposition \ref{alternative definition} are reformulated as follows:
\begin{align*}
\phi_u &= \{ v \in V(H) \ |\ (u, v) \in \phi \}, \\
\overline{\phi} &= \{ u \in V(G) \ |\ \phi_u \cap U \neq \emptyset \}.
\end{align*}

Let $\phi \in I(G[H;U])$ be an independent set of $G[H;U]$. Then, we have the followings.
\begin{itemize}
\item For any $u \in V(G)$ and $v_1, v_2 \in V(H)$, suppose that $v_1, v_2 \in \phi_u$. Then we have $(u, v_1)(u, v_2) \notin E(G[H;U])$ since $(u,v_1), (u, v_2) \in \phi$ and $\phi$ is an independent set of $G[H;U]$. By the definition of $E(G[H;U])$, we obtain $v_1 v_2 \notin E(H)$. Therefore, $\phi_u$ is an independent set of $H$, namely $\phi_u \in I(H)$.
\item For any $u_1, u_2 \in \overline{\phi}$, there exist $w_1, w_2 \in U$ such that $(u_1, w_1), (u_2, w_2) \in \phi$. Then $(u_1, w_1)(u_2, w_2) \notin E(G[H;U])$ since $\phi$ is an independent set of $G[H;U]$. By the definition of $E(G[H;U])$, we obtain $u_1 u_2 \notin E(G)$. Therefore, $\overline{\phi}$ is an independent set of $G$, namely $\overline{\phi} \in I(G)$.
\end{itemize}
So, by Proposition \ref{alternative definition}, $\phi$ is a simplex of $\polyind$.

Conversely, let $\psi$ be a simplex of $\polyind$. It follows from Proposition \ref{alternative definition} that  we have
$\psi_u \in I(H)$ for any $u \in V(G)$ and $\overline{\psi} \in I(G)$. Then, for any $(u_1, v_1), (u_2, v_2) \in \psi$, we have the followings.
\begin{itemize}
\item If $u_1 = u_2$, then $v_1 v_2 \notin E(H)$ since $v_1, v_2 \in \psi_{u_1}$ and $\psi_{u_1}$ is an independent set of $H$. Furthermore, $u_1 = u_2$ implies that $u_1 u_2 \notin E(G)$ since $G$ has no loops. Therefore, we get $(u_1, v_1)(u_2, v_2) \notin E(G[H;U])$.
\item If $u_1 \neq u_2$ and $v_1, v_2 \in U$, then $u_1 u_2 \notin E(G)$ since $u_1, u_2 \in \overline{\psi}$ and $\overline{\psi}$ is an independent set of $G$. Therefore, we get $(u_1, v_1)(u_2, v_2) \notin E(G[H;U])$.
\item If $u_1 \neq u_2$ and $v_1 \notin U$ or $v_2 \notin U$, then $(u_1, v_1)(u_2, v_2) \notin E(G[H;U])$.
\end{itemize}
So, we conclude that $(u_1, v_1)(u_2, v_2) \notin E(G[H;U])$. Thus, $\psi$ is an independent set of $G[H;U]$, namely $\psi \in I(G[H;U])$.
\end{proof}

\begin{example}
Let $G$ be a graph with at least one edge and $H$ be a graph which is not a complete graph. 
Vander Meulen and Van Tuyl \cite[Theorem 2.3]{VandermeulenVantuyl17} proved that $I(G[H])$ is not shellable.
We can deduce this result from Example \ref{lexicographic product}, Proposition \ref{independence complex} and Corollary \ref{dimension corollary} since $I(G)$ is not a simplex and $I(H)$ is not $0$-dimensional.
\end{example}

Now we are ready to prove Theorem \ref{main theorem}.
\begin{proof}[Proof of Theorem \ref{main theorem}]
First, we show that (1) implies (2). By Proposition \ref{independence complex}, (1) is equivalent to the condition that $\polyindv$ is pure and shellable for any graph $G$. It follows from Theorem \ref{pure} that $I(H)$ is pure, namely $H$ is well-covered. Now consider the cycle on $4$ vertices $C_4$, namely the graph defined by
\begin{align*}
V(C_4) =\{1,2,3,4\},&  &E(C_4)= \{12, 23, 34, 41\}.
\end{align*}
Since $I(C_4) = \langle \{1,3\}, \{2,4\} \rangle$ and $I(C_4[H, H \setminus \{v_0\}])$ is shellable, Claim \ref{nonshellable} indicates that $\facet{I(H \setminus \{v_0\})} \subset \facet{I(H)}$. Hence, for any maximal independent set $\tau$ of $H \setminus \{v_0\}$, we get that $\tau \cup \{v_0\}$ is not an independent set of $H$. Namely there exists $v \in \tau$ such that $v_0 v \in E(H)$. Therefore, by Theorem \ref{necessary1}, both $I(H)$ and $I(H \setminus \{v_0\})$ are shellable.

Next, we deduce (1) from (2). By the conditions in (2), $I(H)$ is pure, both $I(H)$ and $I(H \setminus \{v_0\})$ are shellable, and $\facet{I(H \setminus \{v_0\})} \subset \facet{I(H)}$. Therefore, it follows from Theorem \ref{pure} that $I(G[H;H\setminus \{v_0\}])$ is pure for any graph $G$ and from Theorem \ref{equivalent} that $I(G[H; H \setminus \{v_0\}])$ is shellable for any graph $G$ which has at least one edge. For graph $G$ which has no edges, it follows from Lemma \ref{join shellable} that
\begin{align*}
\polyindv = \underbrace{I(H) * \cdots *I(H)}_{|V(G)|}
\end{align*}
is shellable since $I(H)$ is shellable. Therefore, we conclude that $\polyindv$ is pure and shellable, namely $G[H;\{v_0\}]$ is well-covered and shellable, for any graph $G$.
\end{proof}

\begin{example}
If $H$ is a complete graph, then $H$ satisfies the condition (2) in Theorem \ref{main theorem}.

An example of $H$ which is not a complete graph is $C_5$, a cycle of length $5$. Let $V(C_5)=\{a, b, c, d, e\}$ and $E(C_5)=\{ab, bc, cd, de, ea\}$. Then 
\begin{align*}
I(C_5) = \langle \{a,c\},\{b,d\},\{c,e\},\{d,a\},\{e,b\} \rangle
\end{align*}
is pure and shellable. Furthermore, 
\begin{align*}
\dl{C_5}{a} = \langle \{b,d\}, \{e,b\}, \{c,e\} \rangle
\end{align*}
is shellable and each facet of $\dl{C_5}{a}$ contains $b$ or $e$, which are adjacent to $a$.
\end{example}


\begin{thebibliography}{10}

\bibitem{Ayzenberg13}
A.~A. Ayzenberg,
\newblock Substitutions of polytopes and of simplicial complexes, and
  multigraded {B}etti numbers,
\newblock {\em Transactions of the Moscow Mathematical Society}, 74:175--202,
  2013.

\bibitem{BahriBenderskyCohenGitler15}
A.~Bahri, M.~Bendersky, F.R. Cohen, and S.~Gitler,
\newblock Operations on polyhedral products and a new topological construction
  of infinite families of toric manifolds,
\newblock {\em Homology, Homotopy and Applications}, 17(2):137--160, 2015.

\bibitem{BjornerWachs96}
Anders Bj{\"{o}}rner and Michelle~L. Wachs,
\newblock Shellable nonpure complexes and posets. {I},
\newblock {\em Transactions of the American Mathematical Society},
  348(4):1299--1327, 1996.

\bibitem{BjornerWachs97}
Anders Bj{\"{o}}rner and Michelle~L. Wachs,
\newblock Shellable nonpure complexes and posets. {II},
\newblock {\em Transactions of the American Mathematical Society},
  349(10):3945--3975, 1997.

\bibitem{Harary69}
Frank Harary,
\newblock {\em Graph Theory},
\newblock Addison-Wesley Publishing Company, 1969.

\bibitem{HibiHigashitaniKimuraOKeefe15}
Takayuki Hibi, Akihiro Higashitani, Kyouko Kimura, and Augustine~B. O'Keefe,
\newblock Algebraic study on {C}ameron-{W}alker graphs,
\newblock {\em Journal of Algebra}, 422:257--269, 2015.

\bibitem{MoradiKhoshahang16}
Somayeh Moradi and Fahimeh Khosh-Ahang,
\newblock Expansion of a simplicial complex,
\newblock {\em Journal of Algebra and Its Applications}, 15(1):1650004, 2016.

\bibitem{VantuylVillarreal08}
Adam {Van Tuyl} and Rafael~H. Villarreal,
\newblock Shellable graphs and sequentially {C}ohen-{M}acaulay bipartite
  graphs,
\newblock {\em Journal of Combinatorial Theory, Series A}, 115:799--814, 2008.

\bibitem{VandermeulenVantuyl17}
Kevin~N. {Vander Meulen} and Adam {Van Tuyl},
\newblock Shellability, vertex decomposability, and lexicographical products of
  graphs,
\newblock {\em Contributions to Discrete Mathematics}, 12(2):63--68, 2017.

\bibitem{Vidaurre18}
Elizabeth Vidaurre,
\newblock On polyhedral product spaces over polyhedral joins,
\newblock {\em Homology, Homotopy and Applications}, 20(2):259--280, 2018.

\bibitem{Woodroofe09}
Russ Woodroofe,
\newblock Vertex decomposable graphs and obstructions to shellability,
\newblock {\em Proceedings of the American Mathematical Society},
  137(10):3235--3246, 2009.

\end{thebibliography}
\end{document}